\documentclass[10pt,a4paper]{amsart}
\usepackage{amssymb}

\usepackage[dvipsnames]{xcolor} 
\colorlet{cite}{LimeGreen!50!Green}
\usepackage{tikz}  
\usetikzlibrary{arrows,positioning}
\tikzset{ 
  baseline=-2.3pt,
  text height=1.5ex, text depth=0.25ex,
  >=stealth,
  node distance=2cm,
  mid/.style={fill=white,inner sep=2.5pt},
}

\usepackage{amsthm, amssymb, amsfonts}	

\usepackage{graphicx,caption,subcaption}

\usepackage{braket}  
\usepackage{fouriernc}  
\usepackage{microtype}

\usepackage[%
  bookmarks=true,			
  unicode=true,			
  pdftitle={Compactifications of Adjoint Orbits and their Hodge Diamonds},		%
  pdfauthor={Brian Callander}{Elizabeth Gasparim},	%
  pdfkeywords={Hodge diamond}{adjoint orbit}{compactification}{Lefschetz fibration},	
  colorlinks=true,		
  linkcolor=Red,			
  citecolor=cite,		
  filecolor=magenta,		
  urlcolor=RoyalBlue			
]{hyperref}				
\usepackage{cleveref}

\newtheoremstyle{mydef}
  {}		
  {}		
  {}		
  {}		
  {\scshape}	
  {. }		
  { }		
  {\thmname{#1}\thmnumber{ #2}\thmnote{ #3}}	

\theoremstyle{plain}	
\newtheorem{theorem}{Theorem} 
\newtheorem{lemma}[theorem]{Lemma} 
\newtheorem{proposition}[theorem]{Proposition}
\newtheorem{corollary}[theorem]{Corollary}

\theoremstyle{mydef} 
\newtheorem{definition}[theorem]{Definition}

\theoremstyle{remark}

\newtheorem{remark}{Remark}
\newtheorem*{question*}{Question}

\crefname{part}{Part}{Parts}
\crefname{chapter}{Chapter}{Chapters}
\crefname{section}{Section}{Sections}
\crefname{theorem}{Theorem}{Theorems}
\crefname{proposition}{Proposition}{Propositions}
\crefname{lemma}{Lemma}{Lemmata}
\crefname{corollary}{Corollary}{Corollaries}
\crefname{definition}{Definition}{Definitions}
\crefname{example}{Example}{Examples}
\crefname{remark}{Remark}{Remarks}
\crefname{notation}{Notation}{Notations}
\crefname{figure}{Figure}{Figures}
\crefname{enumi}{Item}{Items}

\newcommand{\abs}[1]{\left\lvert #1 \right\rvert}

\newcommand{\homog}{\text{hom}}

\DeclareMathOperator{\Coh}{Coh}

\DeclareMathOperator{\Proj}{Proj}

\DeclareMathOperator{\id}{id}

\DeclareMathOperator{\Lag}{Lag}
\DeclareMathOperator{\reg}{reg}
\DeclareMathOperator{\crit}{\mathfrak{crit}}

\DeclareMathOperator{\diag}{Diag}

\DeclareMathOperator{\Diag}{Diag}

\author{B. Callander, E. Gasparim, L. Grama,  L. A. B.  San Martin}
\address{Addresses: Callander, Grama, San Martin - Imecc -
Unicamp, Depto. de Matem\'{a}tica. Rua S\'{e}rgio Buarque de Holanda,
651, Cidade Universit\'{a}ria Zeferino Vaz. 13083-859 Campinas - SP, Brasil. \qquad
Gasparim - Depto. Matem\'aticas, Universidad Cat\'olica del Norte, Antofagasta, Chile. \qquad
 emails:
{ briancallander@gmail.com,   etgasparim@gmail.com, linograma@gmail.com, smartin@ime.unicamp.br.}}

\title{Symplectic Lefschetz fibrations \\
from a  Lie theoretical viewpoint}

\begin{document}
\maketitle

\begin{abstract} This is an announcement of  results proved in \cite{GGS1}, \cite{GGS2},  \cite{C}, and \cite{CG}  where
methods from  Lie theory were used as  new tools for the study of symplectic Lefschetz fibrations. 
\end{abstract}

\tableofcontents

\section*{Motivation} 

A   motivation for studying symplectic Lefschetz fibrations is that, in  nice cases, they occur as mirror partners of complex varieties. 
In fact, given a complex variety $Y$, the Homological Mirror Symmetry (HMS) conjecture of Kontsevich  \cite{Ko} predicts the existence of a symplectic mirror partner $X$ with a superpotential $W \colon X \to \mathbb C$. For Fano varieties, the statement of the HMS includes the following:
{\it  The category of A-branes $D(\Lag (W ))$ is equivalent to the derived category of B-branes (coherent sheaves) $D^b (\Coh(X))$ on $X$.}
Here $D(\Lag (W))$ is the  directed Fukaya--Seidel category of vanishing cycles for the symplectic manifold $X$ and $D^b (\Coh(Y))$ is the bounded derived category of coherent sheaves on $Y$. 
An exciting feature of the conjecture is that the A-side is symplectic whereas the B-side is algebraic, and therefore the conjecture provides a dictionary between the two types of geometry -- algebraic and symplectic -- the mirror map interchanging vanishing cycles on the symplectic side with coherent sheaves on the algebraic side.

 HMS  has been described in several cases: elliptic curves  \cite{PZ}, curves of genus two  \cite{Se1}, curves of higher genus  \cite{E}, punctured spheres  \cite{AAEKO}, weighted projective planes and del-Pezzo surfaces  \cite{AKO1}, \cite{AKO2}, quadrics and intersection of two quadrics  \cite{S}, the four torus \cite{AbS}, Calabi--Yau hypersurfaces in projective space  \cite{Sh}, toric varieties  \cite{Ab}, 
Abelian varieties  \cite{F}, hypersurfaces in toric varieties \cite{AAK}, 
  varieties of general type \cite{GKR}, and   non-Fano toric varieties  \cite{BDFKK}.
Nevertheless, the HMS conjecture remains open in most cases.

The B-side of the conjecture is better understood in the sense that a lot is known about the category of coherent sheaves on algebraic varieties. In particular, in the Fano and general type cases,   the famous reconstruction theorem of Bondal and Orlov says that you can recover the variety from its derived category of coherent sheaves \cite{BO}.
In contrast the   A-side is rather  mysterious. The intent of this paper is to contribute to the understanding of LG models and subsequently to their categories of vanishing cycles.
Using Lie theory, we construct LG models $(\mathcal{O}(H_0),f_H))$, where $\mathcal{O}(H_0)$ is the adjoint orbit of a complex semisimple  Lie group and $f_H$ is the height 
function with respect to an element of the  Cartan subalgebra (see Theorem \ref{thm-prin-1}). 

Even though we had HMS as an encouragement to pursue our work, we do not attempt to prove any instance of it, rather we  endeavour to contribute to the understanding of  the A-side of the conjecture by describing examples of symplectic Lefschetz fibrations in arbitrary dimensions. We calculate the  directed Fukaya--Seidel category in the first nontrivial example, namely the adjoint orbit of ${\mathfrak{sl}(2,\mathbb{C})} $. For the case of ${\mathfrak{sl}(3,\mathbb{C})} $ orbits, we discuss (the wild) variations of Hodge diamonds depending on choices of  compactifications for our Lefschetz fibrations.

Acknowledgement: It is a pleasure to thank Denis Auroux for suggesting  corrections and improvements to   the text. 

E.~Gasparim and L.~Grama were supported by Fapesp under grant numbers 2012/10179-5 and 2012/21500-9, respectively.
L.~A.~B.~San Martin was supported by CNPq grant n$^{\mathrm{o}}$ 304982/2013-0 and FAPESP grant n$^{\mathrm{o}}$ 2012/17946-1.

\section{Definitions}

  \begin{definition}
    A \emph{holomorphic Morse function} on a manifold $X$ is a holomorphic
    function $f \colon X \to \mathbb{P}^1$  (or  $f \colon X \to \mathbb C$) which has only non-degenerate
    critical points.  
  \end{definition}

  \begin{definition}
    \label[definition]{def:TLF}
    Let $X$ be a complex manifold of dimension $n$ and $f \colon X \to
    \mathbb{P}^1$  (or  $f \colon X \to \mathbb C$)
    a  surjective holomorphic fibration.  We say that $f$ is a \emph{topological
    Lefschetz fibration} 
    if
    \begin{enumerate}
      \item \label{item:TLFCritPts} there are finitely many critical points $p_1, \dotsc, p_k$, 
	and $f (p_i) \neq f (p_j)$ for $i \neq j$;
      \item \label{item:TLFLocMorse} for each critical point $p$, there are
        complex neighbourhoods $p \in U \subset X $,
        $f (p) \in V \subset \mathbb{P}^1$ on which $f_{\vert U}$ is represented by the holomorphic Morse function
	\[
      f_{\vert U} (z_1, \dots, z_n) = z_1^2 + \dotsb + z_n^2,
	\]
    and such that $\crit f \cap U = \set{p}$; and
  \item \label{item:TLFlocTriv} the restriction $f_{\reg} := f
    \vert_{X - \bigcup X_i}$ to the complement of the singular fibres
    $X_i$ is a locally trivial fibre bundle.
    \end{enumerate}
  \end{definition}
  
    \begin{definition}
    \label[definition]{def:SLF}
    Let $X$ be a complex manifold and $\omega$ a symplectic form making
    $(X, \omega)$ into a symplectic manifold.  We
    say that a topological Lefschetz fibration is a \emph{symplectic Lefschetz
    fibration} 
    if
    \begin{enumerate}
      \item \label{item:SLFFibres} the smooth part of any fibre is a symplectic
	submanifold of $(X, \omega)$; and 
      \item \label{item:SLFTanCones} for each critical point $p_i$, the form $\omega_{p_i}$ is
	non-degenerate on the tangent	cone of $X_i$ at $p_i$. 
    \end{enumerate}
  \end{definition}

\section{Non-examples}

%

\begin{proposition}
Let $M$ be a compact complex manifold with odd Euler characteristic, then $M$ does not fibre over $\mathbb P^1$. 
\end{proposition}

\begin{proof}
The Euler characteristic is multiplicative, that is, for such a fibration we would have $\chi(M)$ =$\chi( \mathbb P^1) \cdot \chi(F)$, but $\chi( \mathbb P^1)=2$. 
\end{proof}

\begin{corollary}\cite[cor 2.19]{C}
  There are no topological   fibrations $f \colon \mathbb P^{2n} \to \mathbb P^1$ for $n > 1$.
\end{corollary}

\begin{proposition}
  There are no algebraic  fibrations $f \colon \mathbb P^{n} \to \mathbb P^1$ for $n > 1$.
\end{proposition}

\begin{proof}
  Fibres of such a fibration would divisors  be in $\mathbb P^n$, but by Bezout theorem any two divisors in $\mathbb P^n$ intersect.
\end{proof}

\section{Good examples in dimension 4}

In 4 (real) dimensions, every symplectic manifold admits a Lefschetz fibration after blowing up finitely many points.  
This is the celebrated result of Donaldson   \cite{Do}:
{\it  For any symplectic 4-manifold $X$, there exists a nonnegative integer $n$ such that the $n$-fold blowup of $X$, topologically $X \#n \mathbb{CP}^2$, admits a Lefschetz fibration $f \colon X \#n \mathbb{CP}^2 \to S^2$. }

In the opposite direction, still in 4D, the existence of a topological Lefschetz fibration on a symplectic manifold guarantees the existence of a symplectic Lefschetz fibration whenever the fibres have genus at least 2
  \cite{GoS}:
{\it  If a 4-manifold $X$ admits a genus $g$ Lefschetz fibration $f \colon X \to \mathbb C$ with $g \ge 2$, then it has a symplectic structure.}

Moreover, the existence of 4D symplectic Lefschetz fibrations with arbitrary fundamental group is guaranteed by
  \cite{ABKP}:
{\it   Let $\Gamma$ be a finitely presentable group with a given finite presentation $a \colon \pi_g \to \Gamma$. 
  Then there exists a surjective homomorphism $b \colon \pi_h \to \pi_g$ for some $h \ge g$ and a symplectic Lefschetz fibration $f \colon X \to S^2$ such that
 the regular fibre of $f$ is of genus $h$,
 $\pi_1 (X) = \Gamma$, and
 the natural surjection of the fundamental group of the fibre of $f$ onto the fundamental group of $X$ coincides with $a \circ b$.}

In general it is possible to construct Lefschetz fibrations in 4D starting with a Lefschetz pencil and then blowing up its base locus (see \cite{Se2}, \cite{Se3} \cite{Go}). 
However, in such cases one needs to fix the indefiniteness of the symplectic form over the exceptional locus by glueing in a correction.
Direct constructions of Lefschetz fibrations in higher dimensions are by and large lacking in the literature. 
This gave us our first motivation to investigate the existence of symplectic Lefschetz fibrations on complex $n$-folds with $n \ge 3$. 
Our construction does not make use of Lefschetz pencils, we construct our symplectic Lefschetz fibrations directly by taking the height functions that come naturally from the Lie theory viewpoint.

\section{A caveat about the norm of  complex Morse functions}

  It is sometimes claimed in the literature  that $\abs{f}^2$ is a real Morse function whenever $f$ is a Lefschetz fibration.
  However, this is  in general  false. We state this fact as a lemma.

\begin{lemma}
Let $X$ be a complex manifold of dimension, $f\colon X \rightarrow \mathbb C$  a Lefschetz fibration and let $p$ be a critical point of $f$.
Then $p$ is a degenerate critical point of $|f-f(p)|^2$.
 
\end{lemma}

\begin{proof} 
  We may choose (complex) charts centred at $p$  such that   with respect to this coordinate system
  $f(z_1, \dotsc, z_n)-f(p) =   \sum_{i=1}^n z_i^2$.
Hence, is it enough to consider the standard Lefschetz fibration 
 $g\colon \mathbb C^n \rightarrow \mathbb C$ given by
$g(z_1, \dotsc, z_n) =   \sum_{i=1}^n z_i^2$,
 and to prove that $0$ is a degenerate critical point of $|g|^2$. In real coordinates
  \begin{align*}
    z := (z_1, \dotsc, z_n)
    &\mapsto
    \sum_{i=1}^n z_i^2
    =
    \sum_{i=1}^n x_i^2 - y_i^2 + 2\sqrt{-1} x_iy_i,
  \end{align*}
  where we have written $z_i = x_i + \sqrt{-1}y_i$.
  Then we have the function
  \begin{align*}
    \abs{g}^2 \colon \mathbb{C}^n
    &\to
    \mathbb{R}
    \\
    z
    &\mapsto
    \left[ \sum_{i=1}^n \left( x_i^2 - y_i^2 \right)\right]^2 
    +
    4 \left[ \sum_{i=1}^n x_iy_i \right]^2
  \end{align*}
  whose differentials are 
  \begin{align*}
    \partial_{x_k} \abs{g}^2
    &=
    4x_k \sum_{i=1}^n \left( x_i^2 - y_i^2 \right)
    +
    8y_k \sum_{i=1}^n x_i y_i,
    \\
    \partial_{y_k} \abs{g}^2
    &=
    -4y_k \sum_{i=1}^n \left( x_i^2 - y_i^2 \right)
    +
    8x_k \sum_{i=1}^n x_i y_i.
  \end{align*}
 Since  $\crit \abs{g}^2 \supset  g^{-1}(0)$,
  any neighbourhood of $0$ contains a non-zero critical point of $\abs{g}^2$ and it follows that $0$ is a degenerate critical point of $\abs{g}^2$.  
\end{proof}

\section{SLFs in higher dimensions via Lie theory}

Let $\mathfrak g$ be a complex semisimple Lie algebra with Cartan subalgebra $\mathfrak h$, and $\mathfrak h_\mathbb R$ the real subspace generated by the roots of $\mathfrak h$. 
An element $H\in \mathfrak{h}$ is  called \textit{regular} if $\alpha \left(
H\right) \neq 0$ for all $\alpha \in \Pi $.

\begin{theorem}\label{thm-prin-1}
  \cite[thm. 3.1] {GGS1}
  Given $H_{0}\in \mathfrak{h}$ and $H\in \mathfrak{h}_{\mathbb{R}}$ with $H$ a regular  element, the potential $f_{H}:\mathcal{O} \left( H_{0}\right) \rightarrow \mathbb{C}$ defined by 
  \[
    f_{H}\left( x\right) = \langle H,x\rangle \qquad x\in \mathcal{O}\left(H_{0}\right) 
  \]
  has a finite number of isolated singularities and defines a Lefschetz fibration; that is to say
  \begin{enumerate}
    \item the singularities are (Hessian) nondegenerate;
    \item if $c_{1},c_{2}\in \mathbb{C}$ are regular values then the level manifolds $f_{H}^{-1}\left( c_{1}\right) $ and $f_{H}^{-1}\left( c_{2}\right) $ are diffeomorphic;
    \item there exists a symplectic form $\Omega $ on $\mathcal{O}\left(H_{0}\right) $ such that the regular fibres are symplectic submanifolds;
    \item each critical fibre can be written as the disjoint union of affine subspaces contained in $\mathcal O \left( H_0 \right)$, each symplectic with respect to $\Omega$.
  \end{enumerate}
\end{theorem}

The full proof is presented in \cite{GGS1},  a particularly interesting component  of the proof states:

\begin{proposition}\cite[prop. 3.3]{GGS1}
  A point $x\in \mathcal O (H_0)$ is  a critical point of $f_{H}$ if and only if $x\in \mathcal{O}\left(
  H_{0}\right) \cap \mathfrak{h}=\mathcal{W}\cdot H_{0}$, where $\mathcal{W}$ 
  is the Weyl group. 
\end{proposition}

Having found a construction of Lefschetz fibrations in higher dimensions, the next
step toward a description of the Fukaya--Seidel  category of the corresponding  LG model  would involve the identification of the Fukaya category 
of a regular fibre. Thus, we studied the diffeomorphism type of  a regular level  for the Lefschetz fibration. This first required the 
realisation of the adjoint orbit as the cotangent bundle of a flag manifold, as we now describe.

We choose  a set of
positive roots $\Pi ^{+}$ and simple roots $\Sigma \subset \Pi ^{+}$
with corresponding Weyl chamber is $\mathfrak{a}^{+}$.
 A subset $\Theta \subset \Sigma $ defines a parabolic subalgebra $%
\mathfrak{p}_{\Theta }$ with parabolic subgroup $P_{\Theta }$ and a flag manifold
$\mathbb{F}_{\Theta }=G/P_{\Theta }$. 
An element 
$H_{\Theta }\in \mathrm{cl}\mathfrak{a}^{+}$ is \textit{characteristic}
for $\Theta \subset \Sigma $ if $\Theta =\{\alpha \in \Sigma :\alpha \left(
H_{\Theta }\right) =0\}$. 
Let
$Z_{\Theta }=\{g\in G:\mathrm{Ad}\left( g\right) H_{\Theta }=H_{\Theta
}\}$ be the centraliser in $G$ of the characteristic element $H_{\Theta }$.

\begin{theorem}\cite[thm. 2.1]{GGS2}\label{iso}
\label{teodifeocotan}The adjoint orbit $\mathcal{O}\left( H_{\Theta }\right)
=\mathrm{Ad}\left( G\right) \cdot H_{\Theta }\approx G/Z_{\Theta }$ of the
characteristic element $H_{\Theta }$ is a $C^{\infty }$ vector bundle over $%
\mathbb{F}_{\Theta }$ isomorphic to the cotangent bundle $T^{\ast }\mathbb{F}%
_{\Theta }$. Moreover, we can write down a diffeomorphism $\iota :\mathrm{Ad%
}\left( G\right) \cdot H_{\Theta }\rightarrow T^{\ast }\mathbb{F}_{\Theta }$
such that

\begin{enumerate}
\item $\iota $ is equivariant with respect to the actions of   $K$, that is,
for all $k\in K$, 
\begin{equation*}
\iota \circ \mathrm{Ad}\left( k\right) =\widetilde{k}\circ \iota
\end{equation*}%
where $K$ is the compact subgroup in the Iwasawa decomposition $G=KAN$, and
$\widetilde{k}$ is the lifting to $T^{\ast }\mathbb{F}_{\Theta }$ (via
the differential) of the action of $k$ on $\mathbb{F}_{\Theta }$.

\item The pullback of the canonical symplectic form on $T^{\ast }\mathbb{F}%
_{\Theta }$ by $\iota $ is the (real) Kirillov--Kostant--Souriaux form on
the orbit.
\end{enumerate}
\end{theorem}

Viewing the orbit as the cotangent bundle of a flag manifold, we can identify the topology of the of the fibres in terms of the topology of the flag.

\begin{corollary}\cite[cor. 4.5]{GGS1}
  The homology of a  regular fibre  coincides  with the homology of 
  $\mathbb{F}_{\Theta }\setminus \mathcal{W}\cdot H_{\Theta}$.
  In particular the middle Betti number is $k-1$ where $k$ is 
  the number of singularities of the fibration (equal to the number of elements in
  $\mathcal W \cdot H_\Theta$).
\end{corollary}

For the case where singular fibres have only one critical point, we have the following corollary.

\begin{corollary}\cite[cor. 5.1]{GGS1}
  \label{cor.sing}
  The homology of the singular fibre though $ w H_\Theta$, $w \in \mathcal{W}$, 
  coincides with that of
  \[
    \mathbb{F}_{H_\Theta} \setminus \set{uH_\Theta \in \mathcal{W}\cdot H_{\Theta} | u\neq w}.
  \]
  In particular, the middle  Betti number of this singular fibre 
  equals $k-2$, where $k$ is the number of singularities of the  fibration $f_H$.
\end{corollary}

\section{Compactifications and their Hodge diamonds}

  Theorem \ref{iso} makes it clear that the  
adjoint orbits considered here  are not compact. 
We want to compare the behaviour of vanishing cycles on $\mathcal O(H_0)$ and on its compactifications.
  Expressing the adjoint orbit as an algebraic variety, we
homogenise its ideal to obtain  a projective variety, which serves as a
compactification.
We calculate  
the sheaf-cohomological dimensions $\dim H^q (X, \Omega^p)$ for the compactified orbits  as well as  for the  fibres of the SLF.
These dimensions shall be called the diamond for the given space; indeed, this is well-known as the Hodge diamond in the non-singular case.
 Calculating such diamonds is computationally 
heavy, so we used Macaulay2.

Choosing a compactification is in general a delicate task:  a
different choice of generators for the defining ideal of the orbit can result in completely
different diamonds of the corresponding compactification, as example \ref{badhd} will show.
To illustrate the behaviour of diamonds, we present some examples of 
 adjoint orbits for $\mathfrak{sl}(3, \mathbb C)$, for which there are three isomorphism types. 
 We chose one that compactifies smoothly and another whose compactification acquires degenerate singularities.  

\subsection{An SLF with 3 critical values}

 In $\mathfrak{sl} (3, \mathbb C)$, consider the orbit $\mathcal O (H_0)$ of
\[
  H_0 = 
  \begin{pmatrix}
    2 & 0 & 0 \\
    0 & -1 & 0 \\
    0 & 0 & -1
  \end{pmatrix}
\]
under the adjoint action.   We fix the
element
\[
  H = 
  \begin{pmatrix}
    1 & 0 & 0 \\
    0 & -1 & 0 \\
    0 & 0 & 0
  \end{pmatrix}
\]
to define the potential $f_H$.  
%
  A general element
$A \in \mathfrak{sl} \left( 3, \mathbb{C} \right)$ has the form
\begin{equation}\label{gen3}A=\left(\begin{matrix}
             x_1 &  y_1&  y_2\\
             z_1 &  x_2 &  y_3\\
             z_2 &  z_3 &  -x_1 - x_2
  \end{matrix}\right)\text{.}
  \end{equation}
In this example, the adjoint orbit $\mathcal{O} (H_0)$ consists of all the
matrices with the minimal polynomial $(A + \id)(A - 2\id)$. So 
the orbit is the affine variety cut out by the 
ideal $I$ generated by the polynomial entries of $(A + \id)(A - 2\id)$.
To obtain a projectivisation of $X$, we first homogenise its ideal $I$ with
respect to a new variable $t$, then
take the corresponding projective variety.  
In this case, the  projective variety
$\overline{X}$ is a smooth compactification of $X$ and has 
 Hodge diamond:
\[
  \begin{array}{ccccccccc}
    &&&& 1 &&&& \\
    &&& 0 && 0 &&& \\
    && 0 && 2 && 0 && \\
    & 0 && 0 && 0 && 0 &  \\
    0 && 0 && 3 && 0 && 0 \\
    & 0 && 0 && 0 && 0 & \\
    && 0 && 2 && 0 && \\
    &&& 0 && 0 &&& \\
    &&&& 1 &&&& 
  \end{array} \text{.}
\]

 We now calculate the Hodge diamond of a compactified regular fibre. 
The potential corresponding to our choice of $H$ is 
 $f_H = x_1 - x_2$.
The critical values of this potential are $\pm 3$ and $0$.
Since all regular fibres of an SLF are isomorphic, it suffices to chose the
regular value $1$.  We then define the regular fibre $X_1$  as the
variety in $\mathfrak {sl} (3, \mathbb C) \cong \mathbb C^8$ corresponding to
the ideal $J$ obtained by summing $I$ with the ideal generated by $f_H-1$. 
We then homogenise $J$ to obtain a projectivisation $\overline{X}_1$ of the
regular fibre $X_1$. The Hodge diamond of $\overline{X}_1$ is:
\[
  \begin{array}{ccccccc}
    &&& 1 &&& \\
    && 0 && 0 && \\
    & 0 && 2 && 0 & \\
    0 && 0 && 0 && 0  \\
    & 0 && 2 && 0 & \\
    && 0 && 0 && \\
    &&& 1 &&& 
  \end{array}\text{.}
\]

\begin{remark}
An interesting feature to observe here is the absence of middle cohomology for the
regular fibre. Suppose that the potential extended to this compactification without degenerate singularities, then 
 because  $f_H$  has singularities, the fundamental lemma of Picard--Lefschetz theory 
would imply that there must exist vanishing cycles, which contradicts the absence of middle homology.
\end{remark}

Generalising this example to the case of $\mathfrak{sl}(n,\mathbb C)$, we obtained:

\begin{proposition}\cite[Prop. 2]{CG} Let $H_0 = \diag (n, -1, \dotsc, -1)$. 
Then the orbit of $H_0$ in  $\mathfrak{sl}(n+1,\mathbb C)$ compactifies holomorphically  to a trivial product.
\end{proposition}

\begin{corollary} \cite[Cor 3]{CG} 
  Choose $H = \diag(1,-1,0, \dotsc, 0) $ and $H_0 = \diag (n, -1, \dotsc, -1)$ in $\mathfrak{sl}(n+1,\mathbb C)$.
  Any extension of the potential $f_H$ to the compactification  $\mathbb P^n\times {\mathbb P^n}^{\ast}$ of the orbit $\mathcal O (H_0)$ cannot be of Morse type; that is, it must have degenerate singularities.
\end{corollary}

\subsection{An SLF with 4 critical values}\label{reg1}
In $\mathfrak{sl}\left( 3,\mathbb{C}\right) $ we take 
\[
H=H_{0}=\left( 
\begin{array}{ccc}
1 & 0 & 0 \\
0 & -1 & 0 \\ 
0 & 0 &  0%
\end{array}%
\right) ,
\]%
which  is regular since it has 3 distinct eigenvalues.
Then $X= \mathcal{O}\left( H_{0}\right) $ is the set of 
matrices in  $\mathfrak{sl}\left( 3,\mathbb{C}\right) $ with eigenvalues $1,0,-1$. 
This set forms a submanifold of real dimension $6$ (a complex threefold).
In this  case $\mathcal{W}\simeq S_3$ acts via
conjugation by permutation matrices. Therefore, the potential 
$f_H=x_1-x_2$ has 6 singularities;
namely, the 6 diagonal matrices with diagonal entries $1,0,-1$.
The four singular values of $f_H$ are $\pm 1, \pm2$.
  Thus, $0$ is a regular value for $f_H$.
Let 
$A \in \mathfrak{sl}(3,\mathbb C)$ be a general element written as in (\ref{gen3}), 
and let $p= \det(A)$, $q=\det(A-\id)$. The ideals $\langle p,q\rangle $ and $\langle p-q,q\rangle $ are clearly 
identical and either of them defines the orbit though $H_0$ as an affine variety in $ \mathfrak{sl}\left( 3,\mathbb{C}\right) $.
Now
$$I = \langle p,q,f_H\rangle \qquad J=\langle p,p-q,f_H\rangle $$ 
are two  identical ideals cutting out the regular fibre $X_0$ over $0$. Let
$I_{\homog}$ and $J_{\homog}$ be the respective homogenisations and 
notice that $I_{\hom}\neq J_{\hom}$, so that  they define distinct projective
varieties, and thus two distinct compactifications 
\begin{align*}
  \overline X_0^I &= \Proj (\mathbb C[ x_1,x_2,y_1,y_2,y_3,z_1,z_2,z_3,t]/I_{\hom}) 
  \quad \text{and}
  \\
  \overline X_0^J &= \Proj(\mathbb C[x_1,x_2,y_1,y_2,y_3,z_1,z_2,z_3,t]/J_{\hom}) 
\end{align*}
of $X_0$.
Their diamonds  are given in figure \ref{badhd}.

\begin{figure}[htp]
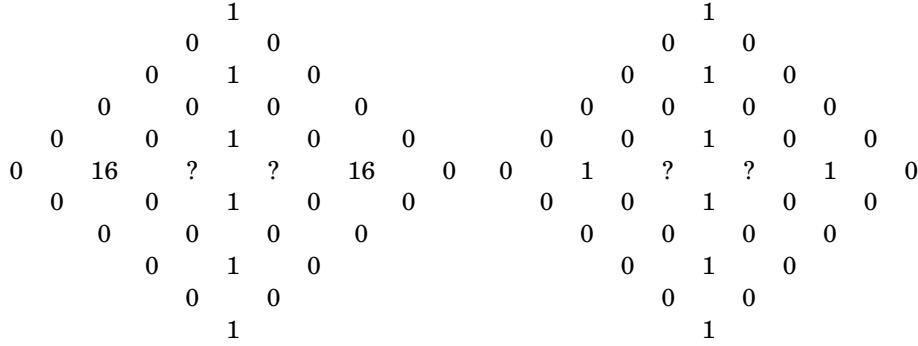

  \begin{subfigure}[b]{0.45\textwidth}
    $\begin{array}{ccccccccccc}
      &&&&& 1 \cr
      &&&& 0 && 0 \cr
      &&& 0 && 1 && 0  \cr
      && 0 && 0 && 0 && 0  \cr
      & 0 && 0 && 1 && 0 && 0  \cr
       0 && 16 && ? && ? && 16 && 0 \cr
      & 0 && 0 && 1 && 0 && 0  \cr
      && 0 && 0 && 0 && 0  \cr
      &&& 0 && 1 && 0  \cr
      &&&& 0 && 0  \cr
      &&&&& 1 \cr
    \end{array}$
    \label{fig:110ur}
  \end{subfigure}
  \qquad
  \begin{subfigure}[b]{0.45\textwidth}
    $\begin{array}{ccccccccccc}
      &&&&& 1 &&&&&    \\
      &&&& 0 && 0 &&&&    \\
      &&& 0 && 1 && 0 &&&    \\
      && 0 && 0 && 0 && 0 &&    \\
      & 0 && 0 && 1 && 0 && 0 &    \\
       0 & & 1 && ? && ? && 1 && 0    \\
      & 0 && 0 && 1 && 0 && 0 &    \\
      && 0 && 0 && 0 && 0 &&    \\
      &&& 0 && 1 && 0 &&&    \\
      &&&& 0 && 0 &&&&    \\
      &&&&& 1 &&&&&    
    \end{array}$
    \label{fig:110sr}
  \end{subfigure}
  \caption{The diamonds of two projectivisations $\overline X_0^I$ (left) and $\overline X_0^J$ (right) of  the regular fibre
  corresponding to $H=H_0=\Diag(1,-1,0)$.
}
  \label{badhd}
\end{figure}

\begin{remark}
The variety $\overline X_0^J$ is an irreducible component of $\overline X_0^I$. Indeed, we find that $I \subset J$
and that $J$ is a prime ideal (whereas $I$ is not).
The discrepancy of values in the middle row
is corroborated by the discrepancy between the expected Euler characteristics of
the compactifications.
\end{remark}

\begin{remark} 
  Macaulay2 greatly facilitates cohomological calculations that
  are unfeasible by hand. However, the memory requirements rise steeply with the
  dimension of the variety. 
  The unknown entries in our diamonds (marked with a `?')
  exhausted the
  48GB of RAM of the computers of our collaborators at IACS Kolkata, 
   without producing an answer.
\end{remark}

\begin{question*}
 This leaves us with the open question of characterising  all the 
compactifications of a given orbit produced by the method of homogenising the defining ideals.
\end{question*}

Nevertheless, once again methods of Lie theory provide us with  sharper tools, and 
we obtain a compactification that is natural from the Lie theory viewpoint.

Let $w_{0}$ be the principal involution of the Weyl group $\mathcal{W}$,
that is, the element of highest length as a product of simple roots. 
For a subset 
$\Theta \subset \Sigma $ we put $\Theta ^{\ast }=-w_{0}\Theta $ and refer to 
$\mathbb{F}_{\Theta ^{\ast }}$ as the flag manifold dual to $\mathbb{F}%
_{\Theta }$. If $H_{\Theta }$ is a characteristic element for $%
\Theta $ then $-w_{0}H_{\Theta }$ is characteristic for $\Theta ^{\ast }$.
Then  the diagonal action of $G$ on the product $\mathbb{F}_{\Theta
}\times \mathbb{F}_{\Theta ^{\ast }}$ as $(g,(x,y))\mapsto (gx,gy)$, $g\in G$%
, $x,y\in \mathbb{F}$  has just one open and dense orbit
which is $G/Z_{\Theta }$.

Let $x_{0}$ be the origin of $\mathbb{F}_{\Theta }$. Since $G$ acts
transitively on $\mathbb{F}_{\Theta}$, all the $G$-orbits of the diagonal action
have the form $G\cdot (x_{0},y)$, with $y\in \mathbb{F}_{\Theta ^{\ast }}$.
Thus, the $G$-orbits are in bijection with 
the orbits through $wy_{0}$, $w\in \mathcal{W}$, where $y_{0}$ is the origin
of $\mathbb{F}_{\Theta ^{\ast }}$.
We obtain:

\begin{proposition}\cite[Prop. 3.1]{GGS2}
The orbit $G\cdot (x_{0},{w}_{0}y_{0})$ is open and dense in $\mathbb{F%
}_{\Theta }\times \mathbb{F}_{\Theta ^{\ast }}$ and identifies to $G/Z_{H}$.
\end{proposition}

\begin{remark}
  Katzarkov, Kontsevich, and Pantev  \cite{KKP}  give three definitions
  of Hodge numbers for Landau--Ginzburg models and conjecture their equivalence. Understanding the relation
  between the diamonds we presented here and those Hodge numbers provides
  a new perspective to our work.
\end{remark}


\begin{thebibliography}{WWW}

\bibitem[Ab]{Ab} Abouzaid, M.\thinspace ; \textit{Morse homology, tropical
geometry, and homological mirror symmetry for toric varieties}. Selecta
Math. (N.S.) \textbf{15} (2009), no. 2, 189--270.

\bibitem[AAEKO]{AAEKO} Abouzaid, M.\thinspace ; Auroux, D.\thinspace ;
Efimov, A.\thinspace ; Katzarkov, L.\thinspace ; Orlov, D.\thinspace ; 
\textit{Homological mirror symmetry for punctured spheres}, J. Amer. Math.
Soc. \textbf{26} (2013), 1051--1083 .

\bibitem[AAK]{AAK} Abouzaid, M.\thinspace ; Auroux, D.\thinspace ;
Katzarkov, L.\thinspace; \textit{Lagrangian fibrations on blowups of toric varieties 
and mirror symmetry for hypersurfacdes}, arXiv:1205:0053.

\bibitem[AbS]{AbS} Abouzaid, M.\thinspace ; Smith, I.\thinspace ; \textit{%
Homological mirror symmetry for the 4-torus}. Duke Math. J. 1\textbf{52}
(2010), no. 3, 373--440.


\bibitem[ABKP]{ABKP} Amor\'{o}s, J.\thinspace ; Bogomolov, F.\thinspace ;
Katzarkov, L.\thinspace ; Pantev, T.\thinspace ; \textit{Symplectic
Lefschetz fibrations with arbitrary fundamental groups}, J. Differential
Geom. \textbf{54} (2000), no. 3, 489--545.

\bibitem[AKO1]{AKO1} Auroux, D.; Katzarkov, L.; Orlov, D.; \textit{Mirror
symmetry for weighted projective planes and their noncommutative deformations%
}, Ann. Math. \textbf{167} (2008), 867--943.

\bibitem[AKO2]{AKO2} Auroux, D.\thinspace ; Katzarkov, L.\thinspace ; Orlov,
D.\thinspace ; \textit{Mirror symmetry for del Pezzo surfaces: vanishing
cycles and coherent sheaves}, Inventiones Math. \textbf{166} (2006), 537--582.

\bibitem[BO]{BO} Bondal, A.\thinspace ; D. Orlov,D.\thinspace ; \textit{%
Reconstruction of a variety from the derived category and groups of
autoequivalences}, Compositio Math. \textbf{125} (2001), no. 3, 327--344.


\bibitem[BDFKK]{BDFKK} Ballard, M.\thinspace ; Diemer, C.\thinspace ;  Favero,D.\thinspace ; Katzarkov, L\thinspace ; Kerr, G.\thinspace ;
\textit{ The Mori program and non-Fano homological mirror symmetry}, arXix:1302.0803. 

  \bibitem[C]{C} Callander, B.\thinspace;
    \textit{Lefschetz Fibrations},
    Master's Thesis,
    Universidade Estadual de Campinas (2013).
    
  \bibitem[CG]{CG} 
    Callander, B.\thinspace; Gasparim, E.\thinspace;
    \textit{ Hodge diamonds and adjoint orbits}, 
    arXiv:1311.1265.
    
 \bibitem[Do]{Do} Donaldson, S. K\thinspace .; \textit{Lefschetz fibrations
in symplectic geometry}, Proceedings of the International Congress of
Mathematicians, Vol. II (Berlin, 1998), Doc. Math. Extra Vol. II (1998),
309--314.
  
  \bibitem[E]{E} Efimov, A.\thinspace ; \textit{Homological mirrror symmetry
for curves of higher genus} Adv. Math. \textbf{230} (2012), no. 2, 493--530.

\bibitem[F]{F} Fukaya, K.\thinspace ; \textit{Mirror symmetry of abelian
varieties and multi-theta functions}. J. Algebraic Geom. \textbf{11} (2002),
no. 3, 393--512.


 


  \bibitem[GGS1]{GGS1}
    Gasparim, E\thinspace;  Grama, L\thinspace;  San Martin, L. A. B.\thinspace;
    \emph{Lefschetz fibrations on adjoint orbits},
    arXiv:1309.4418.

  \bibitem[GGS2]{GGS2}
    Gasparim, E\thinspace;  Grama, L\thinspace;  San Martin, L. A. B.\thinspace;
      \emph{Adjoint orbits of semisimple Lie groups and  Lagrangian submanifolds}, arXiv:1401.2418.

\bibitem[GKR]{GKR} Gross, M.\thinspace; Katzarkov, L.\thinspace; Rudatt, H.\thinspace;
\textit{Towards mirror symmetry for varieties of general type}, arXiv:1202:4042.


    
    \bibitem[Go]{Go} Gompf, R. E.\thinspace ; \textit{Symplectic structures from Lefschetz
pencils in high dimensions}, Geometry \& Topology Monographs \textbf{7}:
Proceedings of the Casson Fest (2004) 267--290.

\bibitem[GoS]{GoS} Gompf, R.\thinspace ; Stipsicz, A.\thinspace ; \textit{An
introduction to 4-manifolds and Kirby calculus}, Graduate Studies in
Mathematics \textbf{20}, American Math. Society, Providence (1999).

\bibitem[KKP]{KKP}
Katzarkov, L.\thinspace; Kontsevich, M.\thinspace;  Pantev, T.\thinspace;
\textit{Bogomolov-Tian-Todorov theorems for Landau-Ginzburg models},
arXiv:1409.5996.


  \bibitem[Ko]{Ko}
    Kontsevich, M.\thinspace;
    \emph{Homological algebra of Mirror Symmetry},
    Proc. International Congress of Mathematicians (Zurich, 1994) Birkh\"auser, Basel (1995) 120--139.

\bibitem[PZ]{PZ} Polishchuk, A.\thinspace; Zaslow, E.\thinspace; \textit{Categorial mirror
symmetry: The elliptic curve}, Adv. Theor. Math. Phys. \textbf{2} (1998)
443--470.

\bibitem[Se1]{Se1} Seidel, P.\thinspace ; \textit{Homological mirror
symmetry for the genus two curve}. J. Algebraic Geom. \textbf{20} (2011),
no. 4, 727--769.


 \bibitem[Se2]{Se2}
    Seidel, P.\thinspace;
    \emph{More about vanishing cycles and mutation}.  Symplectic Geometry and Mirror Symmetry, World Scientific, 2001, 429--465.

  \bibitem[Se3]{Se3}
    Seidel, P.\thinspace;
    \emph{Fukaya categories and Picard-Lefschetz theory},
    Zurich Lectures in Advanced Mathematics, European Math. Soc., Zurich (2008).

\bibitem[Sh]{Sh} Sheridan, N.\thinspace ; \textit{Homological Mirror
Symmetry for Calabi--Yau hypersurfaces in projective space}, arXiv:1111.0632.

\bibitem[S]{S} Smith, I.\thinspace ; \textit{Floer cohomology and pencils of
quadrics}. Invent. Math. \textbf{189} (2012), no. 1, 149--250.


\end{thebibliography}
\end{document}